\documentclass[12pt]{article}
\title{Henselianity in NIP $\Ff_p$-algebras}
\author{Will Johnson}

\usepackage{amsmath, amssymb, amsthm}    	
\usepackage{fullpage} 	
\usepackage{amscd}
\usepackage{hyperref}
\usepackage[all]{xy}
\usepackage{centernot}

\DeclareMathOperator*{\forkindep}{\raise0.2ex\hbox{\ooalign{\hidewidth$\vert$\hidewidth\cr\raise-0.9ex\hbox{$\smile$}}}}

\newcommand{\Sh}{\mathrm{Sh}}

\newcommand{\Frac}{\operatorname{Frac}}

\newcommand{\Spec}{\operatorname{Spec}}

\newtheorem{theorem}{Theorem}[section] 
\newtheorem{lemma}[theorem]{Lemma}

\newtheorem{corollary}[theorem]{Corollary}
\newtheorem{fact}[theorem]{Fact}

\newtheorem{proposition}[theorem]{Proposition}
\newtheorem*{theorem-star}{Theorem}
\newtheorem*{lemma-star}{Theorem}
\newtheorem*{conjecture-star}{Conjecture}

\theoremstyle{definition}
\newtheorem{definition}[theorem]{Definition}

\newtheorem{remark}[theorem]{Remark}
\newtheorem{claim}[theorem]{Claim}

\newtheorem*{acknowledgment}{Acknowledgments}

\newcommand{\Nn}{\mathbb{N}}

\newcommand{\Ff}{\mathbb{F}}

\newcommand{\mm}{\mathfrak{m}}
\newcommand{\pp}{\mathfrak{p}}
\newcommand{\qq}{\mathfrak{q}}

\newenvironment{claimproof}[1][\proofname]
               {
                 \proof[#1]
                 
               }
               {
                 \endproof
               }

\begin{document}
\maketitle

\begin{abstract}
  We prove an assortment of results on (commutative and unital) NIP
  rings, especially $\Ff_p$-algebras.  Let $R$ be a NIP ring.  Then
  every prime ideal or radical ideal of $R$ is externally definable,
  and every localization $S^{-1}R$ is NIP.  Suppose $R$ is
  additionally an $\Ff_p$-algebra.  Then $R$ is a finite product of
  Henselian local rings.  Suppose in addition that $R$ is integral.
  Then $R$ is a Henselian local domain, whose prime ideals are
  linearly ordered by inclusion.  Suppose in addition that the residue
  field $R/\mm$ is infinite.  Then the Artin-Schreier map $R \to R$ is
  surjective (generalizing the theorem of Kaplan, Scanlon, and Wagner
  for fields).
\end{abstract}

\section{Introduction}

The class of NIP theories has played a major role in contemporary
model theory.  See \cite{NIPguide} for an introduction to NIP.  In
recent years, much work has been done on the problem of classifying
NIP fields and NIP rings.  A conjectural classification of NIP fields
has emerged through work of Anscombe, Halevi, Hasson, and Jahnke
\cite{halevi-hasson-jahnke, NIP-char}, and partial results towards
this conjectural classification have been obtained by the author in
the setting of finite dp-rank \cite{wjdpmin-published, prdf1b, prdf6}.

NIP fields are closely connected to NIP valuation rings.  Conjecturally,
\begin{itemize}
\item Every NIP valuation ring is Henselian.
\item Every infinite NIP field is elementarily equivalent to
  $\Frac(R)$ for some NIP non-trivial valuation ring $R$.
\end{itemize}
These conjectures form the basis for the proposed classification of
NIP fields \cite{NIP-char}, and are known to hold assuming finite
dp-rank \cite{prdf6}.  Additionally, the henselianity conjecture is
known in positive characteristic: if $R$ is a NIP valuation ring and
$\Frac(R)$ has positive characteristic, then $R$ is Henselian
\cite[Theorem~2.8]{prdf1a}.

More generally, one would like to understand (commutative) NIP rings,
especially NIP integral domains.  A first step in this direction is
the recent work of Halevi and d'Elb\'ee on dp-minimal integral domains
\cite{halevi-delbee}.  Among other things, they show that if $R$ is a
dp-minimal integral domain, then $R$ is a local ring, the prime ideals
of $R$ are a chain, the localization of $R$ at any non-maximal prime
is a valuation ring, and $R$ is a valuation ring whenever its residue
field is infinite.

In the present paper, we consider a NIP integral domain $R$ such that
$\Frac(R)$ has positive characteristic.  By analogy with
\cite{halevi-delbee}, we show that $R$ is a local ring whose primes
ideals are linearly ordered by inclusion.  Generalizing the earlier
henselianity theorem for valuation rings, we show that $R$ is a
Henselian local ring.  These results may help to extend the work of
Halevi and d'Elb\'ee to ``positive characteristic'' NIP integral
domains.


\subsection{Main results}
All rings are assumed to be commutative and unital.  In
Section~\ref{sec1} we consider a general NIP ring $R$.  Our main
results are the following:
\begin{itemize}
\item Any localization $S^{-1}R$ is interpretable in the Shelah
  expansion $R^\Sh$, and is therefore NIP
  (Theorem~\ref{localizations}).
\item Any radical ideal in $R$ is externally definable (Theorem~\ref{radical}).
\end{itemize}
In Section~\ref{sec2}, we restrict to the case where $R$ is an
$\Ff_p$-algebra, and obtain significantly stronger results:
\begin{itemize}
\item $R$ is a finite product of Henselian local rings (Theorem~\ref{prod-of-hens}).
\item If $R$ is an integral domain, then $R$ is a Henselian local
  domain (Theorem~\ref{one-hens}), and the prime ideals of $R$ are
  linearly ordered by inclusion (Theorem~\ref{linear}).
\item If $R$ is a local integral domain with maximal ideal $\mm$ and
  $R/\mm$ is infinite, then the Artin-Schreier map $R \to R$ is
  surjective (Theorem~\ref{ksw++}).
\end{itemize}
The henselianity results generalize \cite[Theorem~2.8]{prdf1a}, which
handled the case where $R$ is a valuation ring.  The surjectivity of
the Artin-Schreier map generalizes a theorem of Kaplan, Scanlon, and
Wagner \cite[Theorem~4.4]{NIPfields}, which handled the case where $R$ is a
field.

\section{General NIP rings} \label{sec1}
\subsection{Finite width}
The \emph{width} of a poset $(P,\le)$ is the maximum size of an
antichain in $P$.  We write $\Spec R$ for the poset of prime ideals in
$R$, ordered by inclusion.  This is an abuse of notation, since we are
forgetting the usual scheme and topology structure on $\Spec R$, and
then adding the poset structure.
\begin{fact} \label{width}
  Let $R$ be a NIP ring.  Then $\Spec R$ has finite width.  Moreover,
  there is a uniform finite bound on the width of $\Spec R'$ for $R'
  \succeq R$.
\end{fact}
Fact~\ref{width} is proved by Halevi and d'Elb\'ee
\cite[Proposition~2.1, Remark~2.2]{halevi-delbee}, who attribute it to
Pierre Simon.

Fact~\ref{width} has a number of useful corollaries, which we shall
use in later sections.  First of all, Dilworth's theorem gives the
following corollary:
\begin{corollary} \label{dilworth}
  If $R$ is a NIP ring, then $\Spec R$ is a finite union of chains.
\end{corollary}
Another trivial corollary of Fact~\ref{width} is the following:
\begin{corollary} \label{fin-man}
  If $R$ is a NIP ring, then $R$ has finitely many maximal ideals and
  finitely many minimal prime ideals.
\end{corollary}
Also, using Beth's implicit definability, we see the following:
\begin{corollary} \label{def-max}
  If $R$ is a NIP ring, then the maximal ideals of $R$ are definable.
\end{corollary}
For completeness, we give the proof.  The proof uses the following
form of Beth's theorem:
\begin{fact} \label{beth}
  Let $M$ be an $L_0$-structure.  Let $L$ be a language extending
  $L_0$ and let $T$ be an $L$-theory.  Suppose there is a cardinal
  $\kappa$ such that for any $M' \succeq M$ there are at most
  $\kappa$-many expansions of $M'$ to a model of $T$.  Then every such
  expansion is an expansion by definitions.
\end{fact}
\begin{proof}[Proof of Corollary~\ref{def-max}]
  Let $L_0$ be the language of rings and $L$ be $L_0 \cup \{P\}$ where
  $P$ is a unary predicate symbol.  Let $T$ be the statement saying
  that $P$ is a maximal ideal, i.e.,
  \begin{gather*}
    \forall x, y : P(x) \wedge P(y) \rightarrow P(x+y) \\
    P(0) \\
    \forall x, y : P(x) \rightarrow P(x \cdot y) \\
    \neg P(1) \\
    \forall x : \neg P(x) \rightarrow \exists y : P(xy-1).
  \end{gather*}
  If $R' \succeq R$, then an expansion of $R'$ to a model of $T$ is
  the same thing as a maximal ideal of $R'$.  The number of such
  maximal ideals is uniformly bounded by Fact~\ref{width}, and
  so Fact~\ref{beth} shows that each such maximal ideal is definable.
\end{proof}
(Of course, there are other more direct algebraic proofs of
Corollary~\ref{def-max}.)

Recall that the Jacobson radical of a ring is the intersection of its
maximal ideals.
\begin{corollary}\label{Jacobson}
  Let $R$ be a NIP integral domain.  Then the Jacobson radical of $R$
  is non-zero.
\end{corollary}
\begin{proof}
  In a domain, the intersection of two non-zero ideals is non-zero.
\end{proof}
\begin{corollary} \label{joke}
  Let $R$ be a NIP integral domain that is not a field.  Let $K =
  \Frac(R)$.  There is a non-trivial, non-discrete Hausdorff field
  topology on $K$ characterized by either of the following:
  \begin{itemize}
  \item The family of sets $\{aR : a \in K^\times\}$ is a neighborhood
    basis of 0.
  \item The set of non-zero ideals of $R$ is a neighborhood basis of
    0.
  \end{itemize}
\end{corollary}
\begin{proof}
  Everything follows formally by \cite[Example~1.2]{PZ}, except
  that we only get a \emph{ring} topology.  It remains to see that the
  map $x \mapsto 1/x$ is continuous.  It suffices to consider
  continuity around $x = 1$.  Let $I$ be a non-zero ideal in $R$.  We
  claim there is a non-zero ideal $I'$ such that if $x \in 1 + I'$,
  then $1/x \in 1 + I$.  Indeed, take $I' = I \cap J$, where $J$ is
  the Jacobson radical.  Suppose $x \in 1 + (I \cap J)$.  Then $x - 1$
  is in every maximal ideal, implying that $x$ is in no maximal
  ideals, so $x \in R^\times$.  Also, $x \in 1 + I$ implies that $1-x
  \in I$, and then $x^{-1}(1-x) \in I$, because $x$ is a unit.  But
  $x^{-1}(1 - x) = x^{-1} - 1$, and so $x^{-1} \in 1 + I$ as desired.
\end{proof}

\subsection{Localizations}
If $M$ is a structure, then $M^\Sh$ denotes the Shelah expansion of
$M$.  If $M$ is NIP, then the definable sets in $M^\Sh$ are exactly
the externally definable sets in $M$, and $M^\Sh$ is NIP
\cite[Proposition~3.23, Corollary~3.24]{NIPguide}.

Say that a collection of sets $\mathcal{C}$ is ``uniformly definable''
in a structure $M$ if $\mathcal{C} \subseteq \{X_a : a \in Y\}$ for
some definable family of sets $\{X_a\}_{a \in Y}$.
\begin{remark} \label{filtered-union}
  Let $M$ be a structure.  Suppose $D = \bigcup_{i \in I} D_i$ is a
  directed union, and the $D_i$ are uniformly definable in $M$.  Then
  $D$ is externally definable.  
\end{remark}
This is well-known in certain circles, but here is the proof for
completeness:
\begin{proof}
  Take some $L(M)$-formula $\phi(x,y)$ such that $D_i = \phi(M,b_i)$
  for some $b_i \in M^y$.  Let $\Sigma(y)$ be the partial type
  \begin{equation*}
    \{\phi(a,y) : a \in D\} \cup \{\neg \phi(a,y) : a \in M^x \setminus D\}.
  \end{equation*}
  Then $\Sigma(y)$ is finitely satisfiable, because for any
  $a_1,\ldots,a_n \in D$ and $e_1,\ldots,e_m \in M^x \setminus D$ we
  can find some $i$ such that $D_i \supseteq \{a_1,\ldots,a_n\}$,
  because the union is directed.  Then $D_i \subseteq D$, so $D_i \cap
  \{e_1,\ldots,e_m\} = \varnothing$.  Thus $b_i$ satisfies the
  relevant finite fragment of $\Sigma(y)$.  By compactness there is a
  realization $b$ of $\Sigma(y)$ in an elementary extension $N \succeq
  M$.  Then $\phi(M,b) = D$, by definition of $\Sigma(y)$, so $D$ is
  externally definable.
\end{proof}
\begin{lemma} \label{saturater}
  Let $R$ be a NIP ring.  Let $S$ be a multiplicative subset.  Then
  there is an externally definable multiplicative subset
  $\overline{S}$ such that the localization $S^{-1}R$ is isomorphic
  (as an $R$-algebra) to $\overline{S}^{-1}R$.
\end{lemma}
\begin{proof}
  For any $x \in R$, let $F_x$ denote the set of $y \in R$ such that
  $y|x$.  Let $\overline{S} = \bigcup_{x \in S} F_x$.  Note that if
  $A$ is a ring and $f : R \to A$ is a homomorphism, then the
  following are equivalent:
  \begin{itemize}
  \item $f(s)$ is invertible for every $s \in S$.
  \item $f(x)$ is invertible for $x,y,s$ with $xy = s$ and $s \in S$.
  \item $f(x)$ is invergible for $x,s$ with $x \in F_s$ and $s \in S$.
  \item $f(x)$ is invertible for $x \in \overline{S}$.
  \end{itemize}
  Therefore $S^{-1}R$ and $\overline{S}^{-1}R$ represent the same
  functor, and are isomorphic.

  It remains to see that $\overline{S}$ is externally definable.  This
  follows by Remark~\ref{filtered-union} because the sets $F_x$ are uniformly definable, and the
  union $\bigcup_{x \in S} F_x$ is a directed union.  Indeed, if $x, y
  \in S$, then $xy \in S$ and $F_{xy} \supseteq F_x \cup F_y$.
\end{proof}
\begin{theorem} \label{localizations}
  Let $R$ be a NIP ring.  Let $S$ be a multiplicative subset.  Then
  the localization $S^{-1}R$ and the homomorphism $R \to S^{-1}R$ are
  interpretable in $R^\Sh$.
\end{theorem}
\begin{proof}
  By Lemma~\ref{saturater}, we may replace $S$ with an externally
  definable set $\overline{S}$, and then the result is clear.
\end{proof}
\begin{corollary} \label{localize-nip}
  Let $R$ be a NIP ring.  Let $S$ be a multiplicative subset.  Then
  the localization $S^{-1}R$ is also NIP.
\end{corollary}
\begin{proof}
  The localization $S^{-1}R$ is interpretable in the NIP structure
  $R^\Sh$.
\end{proof}
Corollary~\ref{localize-nip} generalizes part of
\cite[Proposition~2.8(2)]{halevi-delbee}, dropping the assumptions
that $S$ is externally definable and $R$ is integral.
\begin{proposition} \label{primes}
  Let $R$ be a NIP ring.  Let $\pp$ be a prime ideal in $R$.  Then
  $\pp$ is externally definable.
\end{proposition}
\begin{proof}
  By Theorem~\ref{localizations}, we can interpret $R \to R_\pp$
  in $R^{\Sh}$.  The maximal ideal of $R_\pp$ is definable in $R_\pp$,
  as the set of non-units.  It pulls back to $\pp$ in $R$.  Therefore
  $\pp$ is definable in $R^{\Sh}$, hence externally definable in $R$.
\end{proof}
Proposition~\ref{primes} generalizes a theorem of Halevi and
d'Elb\'ee, who proved that (certain) prime ideals in dp-minimal
domains are externally definable \cite[Lemma~3.3]{halevi-delbee}.
\begin{theorem} \label{radical}
  Let $R$ be a NIP ring.  Let $I$ be a radical ideal in $R$.  Then
  $I$ is externally definable.
\end{theorem}
\begin{proof}
  By Corollary~\ref{dilworth}, we can cover the set $\Spec R$ of prime
  ideals in $R$ with finitely many chains $\mathcal{C}_1, \ldots,
  \mathcal{C}_n$.  The ideal $I$ is an intersection of prime ideals.
  Let $\pp_i$ be the intersection of the prime ideals $\pp \in
  \mathcal{C}_i$ with $\pp \supseteq I$.  An intersection of a chain
  of prime ideals is prime, so $\pp_i$ is prime.  Then $I$ is a finite
  intersection $\bigcap_{i = 1}^n \pp_i$.  Each $\pp_i$ is externally
  definable by Proposition~\ref{primes}.
\end{proof}
\begin{corollary}
  Let $R$ be a NIP ring.  Let $I$ be a radical ideal.  The quotient
  $R/I$ is NIP.
\end{corollary}
\begin{proof}
  The quotient $R/I$ is interpretable in the NIP structure $R^{\Sh}$.
\end{proof}

\subsection{Automatic connectedness}
If $G$ is a definable or type-definable group, then $G^{00}$ is the
smallest type-definable group of bounded index in $G$.  In a NIP
context, $G^{00}$ always exists, and is type-definable over whatever
parameters define $G$ \cite[Proposition 6.1]{goo}
\begin{proposition} \label{connecter}
  Let $R$ be a NIP ring.  Suppose that $R/\mm$ is infinite for every
  maximal ideal $\mm$ of $R$.
  \begin{enumerate}
  \item \label{first} If $I$ is a definable ideal of $R$, then $I = I^{00}$.
  \item \label{second} If $R$ is a domain and $K = \Frac(R)$ and if
    $I$ is a definable $R$-submodule of $K$, then $I = I^{00}$.
  \end{enumerate}
  In particular, in either case, $I$ has no definable proper subgroups
  of finite index.
\end{proposition}
\begin{proof}
  We may assume $R$ is a monster model, i.e., $\kappa$-saturated for
  some big cardinal $\kappa$.  ``Small'' will mean ``cardinality less
  than $\kappa$'', and ``large'' will mean ``not small.''

  Let $\mm_1,\ldots,\mm_n$ be the maximal ideals of $R$.  By
  Corollary~\ref{fin-man} there are only finitely many, and by
  Corollary~\ref{def-max} they are all definable.  The quotients
  $R/\mm_i$ are infinite, hence large.  Therefore every simple
  $R$-module is large.  Every non-trivial $R$-module has a simple
  subquotient, so every non-trivial $R$-module is large.

  Now suppose $I$ is a definable ideal.  If $a \in R$, then the map $I
  \to I$ sending $x$ to $ax$ must map $I^{00}$ into $I^{00}$.  Indeed,
  if we let $J = \{x \in I : ax \in I^{00}\}$, then $J$ is a
  type-definable subgroup of $I$ of bounded index, so $J \supseteq
  I^{00}$.  Thus we see that for any $a \in R$, we have $aI^{00}
  \subseteq I^{00}$.  In other words, $I^{00}$ is an ideal.  The
  quotient $I/I^{00}$ is an $R$-module.  By definition of $G^{00}$,
  the quotient $I/I^{00}$ is small.  We saw that non-trivial
  $R$-modules are large, so $I/I^{00}$ must be trivial, implying $I =
  I^{00}$.  This proves (\ref{first}), and (\ref{second}) is similar.
\end{proof}

\section{NIP $\Ff_p$-algebras} \label{sec2}
\subsection{A variant of the Kaplan-Scanlon-Wagner theorem}
In \cite[Theorem~4.4]{NIPfields}, Kaplan, Scanlon, and Wagner show
that if $K$ is an infinite NIP field of characteristic $p > 0$, then
the Artin-Schreier map $x \mapsto x^p - x$ is a surjection from $K$
onto $K$.  The same idea can be applied to certain local rings, as we
will see in Theorem~\ref{ksw++} below.

Before proving the theorem, we need some (well-known) lemmas on
additive polynomials.  Fix a field $K$ of characteristic $p$.  If $c
\in K$, define
\begin{equation*}
  g_c(x) = x^p - c^{p-1}x.
\end{equation*}
The polynomial $g_c(x)$ defines an additive homomorphism from $K$ to
$K$.  If $V$ is a finite-dimensional $\Ff_p$-linear subspace of $K$ (i.e., a finite subgroup of $(K,+)$), define
\begin{equation}
  f_V(x) = \prod_{a \in V}(x-a).
\end{equation}
We will see shortly that $f_V$ is an additive homomorphism.
\begin{lemma}\label{lam1}
  If $c \in K$ is non-zero, then $g_c(x) = f_{\Ff_p \cdot c}(x)$.  In
  particular, $f_{\Ff_p \cdot c}(x)$ is an additive homomorphism.
\end{lemma}
\begin{proof}
  Note that $g_c(c) = 0$.  Therefore, $\ker g_c$ contains the subgroup
  generated by $c$, which is $\Ff_p \cdot c$.  Since $g_c$ is monic of
  degree $p$, and $|\Ff_p \cdot c| = p$, we must have
  \begin{equation*}
    g_c(x) = \prod_{a \in \Ff_p \cdot c} (x - a) = f_{\Ff_p \cdot
      c}(x). \qedhere
  \end{equation*}
\end{proof}
\begin{lemma}\label{lam2}
  Suppose $V_1 \subseteq V_2$ are finite-dimensional subspaces of $K$,
  with $\dim V_2 = \dim V_1 + 1$.  Suppose $f_{V_1}$ is an additive
  homomorphism on $K$.  Then there is $c \in f_{V_1}(V_2)$ such that
  $f_{V_2} = g_c \circ f_{V_1}$, and in particular $f_{V_2}$ is an
  additive homomorphism on $K$.
\end{lemma}
\begin{proof}
  Take $a \in V_2 \setminus V_1$ and let $c = f_{V_1}(a)$.  Let $h =
  g_c \circ f_{V_1}$.  Then $h$ is an additive homomorphism on $K$,
  and it suffices to show that $h = f_{V_2}$.  Note that if $x \in
  V_1$, then $h(x) = g_c(f_{V_1}(x)) = g_c(0) = 0$, since $f_{V_1}$
  vanishes on $V_1$.  Additionally, $h(a) = g_c(f_{V_1}(a)) = g_c(c) =
  0$.  Therefore the kernel of $h$ contains $V_1$ as well as $a$.  It
  therefore contains the group they generate, which is $V_1 + \Ff_p
  \cdot a = V_2$.  If $d = \dim V_1$, then $|V_1| = p^d$ and $|V_2| =
  p^{d+1}$.  The polynomial $f_{V_1}$ is a monic polynomial of degree
  $p^d$, and $g_c$ is a monic polynomial of degree $p$.  Therefore the
  composition $h$ is a monic polynomial of degree $p^{d+1}$.  We have
  just seen that $h$ vanishes on the set $V_2$ of size $p^{d+1}$, so
  $h(x)$ must be $\prod_{u \in V_2} (x - u) = f_{V_2}(x)$.
\end{proof}
\begin{lemma}\label{lam3}
  If $V$ is a finite-dimensional subspace of $K$, then $f_V$ is an
  additive homomorphism with kernel $V$.
\end{lemma}
\begin{proof}
  The fact that $f_V$ is an additive homomorphism follows by induction
  on $\dim V$ using Lemma~\ref{lam2}.  The fact that $\ker f_V = V$ is
  immediate from the definition of $f_V$.
\end{proof}
We now can prove our desired theorem on NIP local domains in positive
characteristic:
\begin{theorem} \label{ksw++}
  Let $p > 0$ be a prime.  Let $R$ be a NIP $\Ff_p$-algebra with the
  following properties: $R$ is a local ring, $R$ is an integral domain
  with maximal ideal $\mm$, and the quotient field $k = R/\mm$ is
  infinite.  Then $x \mapsto x^p - x$ is a surjection from $R$ onto
  $R$.
\end{theorem}
\begin{proof}
  Let $K = \Frac(R)$.  Note that if $V$ is a finite-dimensional
  $\Ff_p$-subspace of $R$, then $f_V(x) \in R[x]$, and if $c \in
  R$, then $g_c(x) \in R[x]$.
  \begin{claim} \label{claim1}
    It suffices to find $c \in R^\times$ such that $g_c(x)$ is a
    surjection from $R$ to $R$.
  \end{claim}
  \begin{claimproof}
    Note that $c^{-p}g_c(cx) = c^{-p}(c^px^p - c^{p-1}cx) = x^p - x$.
    The maps $x \mapsto cx$ and $x \mapsto c^{-p}x$ are bijections on
    $R$, so if $g_c$ is surjective then so is $g_1(x) = x^p - x$.
  \end{claimproof}
  For any $c \in R$, the polynomial $g_c(x)$ defines an additive map
  $R \to R$, whose image $g_c(R)$ is an additive subgroup of $R$.  Let
  $\mathcal{G} = \{g_c(R) : c \in R\}$.  By the Baldwin-Saxl theorem
  for NIP groups, there is some integer $n$ such that if $G_1, \ldots,
  G_n \in \mathcal{G}$, then there is some $i$ such that \[G_i
  \supseteq G_1 \cap \cdots \cap G_{i-1} \cap G_{i+1} \cap \cdots \cap
  G_n.\] Fix such an $n \ge 2$.

  The residue field $k$ is infinite, so we can find $\Ff_p$-linearly
  independent $\alpha_1,\ldots,\alpha_n \in k$.  Take $a_i \in R$
  lifting $\alpha_i \in k$.  Note $\alpha_i \ne 0$, so $a_i \notin
  \mm$, and therefore $a_i \in R^\times$.  Also note that the elements
  $\{a_1,\ldots,a_{n-1}\}$ are $\Ff_p$-linearly independent in $K$.

  Let $[n] = \{1,\ldots,n\}$.  If $S \subseteq [n]$ and $i \in [n]$,
  we write $S \cup i$ and $S \setminus i$ as abbreviations for $S \cup
  \{i\}$ and $S \setminus \{i\}$.  Even worse, we sometimes abbreviate
  $\{i\}$ as $i$.

  For $S \subseteq [n]$, let $V_S$ be the
  $\Ff_p$-linear span of $\{a_i : i \in S\}$.  Then $V_S$ has
  dimension $|S|$.  Let
  \begin{equation*}
    f_S(x) := f_{V_S}(x) = \prod_{a \in V_S} (x - a).
  \end{equation*}
  This is a monic polynomial in $R[x]$.  By Lemma~\ref{lam3} $f_S(x)$
  induces an additive homomorphism $K \to K$, and therefore an
  additive homomorphism $R \to R$.

  Note that $f_i(x) = f_{V_i}(x) = f_{\Ff_p \cdot a_i}(x) =
  g_{a_i}(x)$ by Lemma~\ref{lam1}.  By Claim~\ref{claim1}, it
  suffices to show that $f_i$ is a surjection from $R$ to $R$, for
  at least one $i$.
    
  If $S \subseteq [n]$ and $i \in [n] \setminus S$, then $V_{S \cup
    i}$ has dimension one more than $V_S$.  By Lemma~\ref{lam2}, there
  is some $c_{S,i} \in f_S(V_{S \cup i})$ such that $g_{c_{S,i}} \circ
  f_S = f_{S \cup i}$.  Let $g_{S,i} := g_{c_{S,i}}$.  Then
  \begin{equation*}
    g_{S,i} \circ f_S = f_{S \cup i}.
  \end{equation*}
  Now $c_{S,i} \in f_S(V_{S \cup i})$, but $f_S(x) \in R[x]$ and $V_{S
    \cup i} \subseteq R$.  Therefore $c_{S,i} \in R$, and $g_{S,i}(x)
  \in R[x]$.
  \begin{claim} \label{claim2}
    If $S \subseteq [n]$ and $i, j$ are distinct elements of $[n]
    \setminus S$, then $c_{S,i}^{p-1} - c_{S,j}^{p-1} \notin \mm$.
  \end{claim}
  \begin{claimproof}
    Otherwise, the two polynomials $g_{S,i}(x)$ and $g_{S,j}(x)$ have
    the same reduction modulo $\mm$.  From the identities $f_{S \cup
      i} = g_{S,i} \circ f_S$ and $f_{S \cup j} = g_{S,j} \circ f_S$,
    it follows that $f_{S \cup i} \equiv f_{S \cup j} \pmod{\mm}$.
    Let $V'_S$ be the $\Ff_p$-linear span of $\{\alpha_i : i \in S\}$,
    or equivalently, the image of $V_S$ under $R \to R/\mm$.  By
    inspection, the reduction of $f_S$ modulo $\mm$ is $\prod_{u \in
      V'_S} (x - u)$.  Since $V'_{S \cup i} \ne V'_{S \cup j}$, it
    follows immediately that $f_{S \cup i}$ and $f_{S \cup j}$ cannot
    have the same reduction modulo $\mm$, a contradiction.
  \end{claimproof}
  Each of the groups $g_{[n] \setminus i, i}(R)$ is in the family
  $\mathcal{G}$.  By choice of $n$, one of the factors in the
  intersection $\bigcap_{i = 1}^n g_{[n] \setminus i, i}(R)$ is
  irrelevant.  Without loss of generality, it is the first factor:
  \begin{equation}
    g_{[n] \setminus 1, 1}(R) \supseteq \bigcap_{i = 2}^n g_{[n]
      \setminus i, i}(R). \label{call-1}
  \end{equation}
  We claim that $f_1(x)$ defines a surjection from $R$ to $R$.  As
  $f_1(x) = g_{a_1}(x)$, this suffices, by Claim~\ref{claim1}.

  Take some $b_1 \in R$.  It suffices to show that $b_1 \in f_1(R)$.
  Take some $b_\varnothing \in K^{alg}$ such that $f_1(b_\varnothing)
  = b_1$.  It suffices to show that $b_\varnothing \in R$.  For $S
  \subseteq [n]$, define $b_S = f_1(b_\varnothing) \in K^{alg}$.
  (When $S = \{1\}$ this recovers $b_1$, and when $S = \varnothing$
  this recovers $b_\varnothing$, to the notation is consistent.)  Note
  that
  \begin{equation}
    g_{S,i}(b_S) = g_{S,i}(f_S(b_\varnothing)) = f_{S \cup
      i}(b_\varnothing) = b_{S \cup i}. \label{call-z}
  \end{equation}
  \begin{claim} \label{claim-w}
    If $1 \in S \subseteq [n]$, then $b_S \in R$.
  \end{claim}
  \begin{claimproof}
    Take a minimal counterexample $S$.  If $S = \{1\}$, then $b_S =
    b_1 \in R$.  Otherwise, take $i \in S \setminus 1$ and let $S_0 =
    S \setminus i$.  By choice of $S$, we have $b_{S_0} \in R$.  Then
    $b_S = g_{S_0,i}(b_{S_0})$.  But $g_{S_0,i}(x) \in R[x]$, so $b_S
    \in R$.
  \end{claimproof}
  In particular, $b_S \in R$ for $S = [n]$, as well as $S = [n]
  \setminus i$ for $i > 1$.  Then
  \begin{equation*}
    b_{[n]} = g_{[n] \setminus i, i}(b_{[n] \setminus i}) \in g_{[n] \setminus i, i}(R)
  \end{equation*}
  for $1 < i \le n$.  By (\ref{call-1}), $b_{[n]} \in g_{[n] \setminus
    1, 1}(R)$.  Take $v \in R$ such that $g_{[n] \setminus 1, 1}(v) =
  b_{[n]}$.  Then $g_{[n] \setminus 1, 1}(v) = b_{[n]} = g_{[n]
    \setminus 1, 1}(b_{[n] \setminus 1})$, and so
  \begin{equation*}
    v - b_{[n] \setminus 1} \in \ker g_{[n] \setminus 1, 1} = \Ff_p
    \cdot c_{[n] \setminus 1, 1} \subseteq R.
  \end{equation*}
  Therefore $b_{[n] \setminus 1} \in R$.  So we see that
  \begin{equation}
    b_{[n] \setminus i} \in R \text{ for all } 1 \le i \le n. \label{call-2}
  \end{equation}
  \begin{claim}
     $b_\varnothing \in R$.
  \end{claim}
  \begin{claimproof}
    Suppose otherwise.  Take $S$ maximal such that $b_S \notin R$.  By
    Claim~\ref{claim-w} and (\ref{call-2}), $S$ is neither $[n]$ nor
    $[n] \setminus i$ for $1 \le i \le n$.  Therefore $[n] \setminus
    S$ contains at least two elements $i, j$.  By choice of $S$, we
    have $b_{S \cup i} \in R$ and $b_{S \cup j} \in R$.  By
    (\ref{call-z}),
    \begin{gather*}
      b_{S \cup i} = g_{S,i}(b_S) = b_S^p - c_{S,i}^{p-1} b_S \\
      b_{S \cup j} = g_{S,j}(b_S) = b_S^p - c_{S,j}^{p-1} b_S.
    \end{gather*}
    Therefore
    \begin{equation*}
      (c_{S,i}^{p-1} - c_{S,j}^{p-1}) b_S = b_{S \cup j} - b_{S \cup i} \in R.
    \end{equation*}
    By Claim~\ref{claim2}, $c_{S,i}^{p-1} - c_{S,j}^{p-1} \in R
    \setminus \mm = R^\times$, and so $b_S \in R$ a contradiction.
  \end{claimproof}
  This completes the proof.  We see that $b_\varnothing \in R$, and so
  $b_1 = f_1(b_\varnothing) \in f_1(R)$.  As $b_1$ was an arbitrary
  element of $R$, it follows that $f_1$ gives a surjection from $R$ to
  $R$.  But $f_1(x) = g_{a_1}(x)$, and $a_1 \in R^\times$ (since its
  residue mod $\mm$ is the non-zero element $\alpha_1$), and so we are
  done by Claim~\ref{claim1}.
\end{proof}

\subsection{Linearly ordering the primes}
\begin{lemma} \label{step-1}
  Let $R$ be an $\Ff_p$-algebra that is integral and has exactly two
  maximal ideals $\mm_1$ and $\mm_2$.  Suppose that $R/\mm_1$ and
  $R/\mm_2$ are infinite.  Then $R$ isn't NIP.
\end{lemma}
The proof uses an identical strategy to \cite[Lemma~2.6]{prdf1a}.
\begin{proof}
  Suppose $R$ is NIP.  By Corollary~\ref{def-max}, $\mm_1, \mm_2$ are
  definable.  Let $K = \Frac(R)$.  Regard the localizations
  $R_{\mm_1}$ and $R_{\mm_2}$ as definable subrings of $K$.  Note that
  $R_{\mm_1} \cap R_{\mm_2} = R$, by commutative algebra.  (If $x \in
  K \setminus R$, then let $I = \{a \in R : ax \in R\}$; this is a
  proper ideal in $R$ so it is contained in some $\mm_i$, and then $I
  \subseteq \mm_i$ means precisely that $x \notin R_{\mm_i}$.)
  \begin{claim} \label{claim-pre}
    If $x \in R$, then the Artin-Schreier roots of $x$ are in $R$.
  \end{claim}
  \begin{claimproof}
    The rings $R_{\mm_1}$ and $R_{\mm_2}$ satisfy the conditions of
    Theorem~\ref{ksw++}.  (The residue field of $R_{\mm_i}$ is
    isomorphic to $R/\mm_i$, hence infinite.)  Therefore, there are $y
    \in R_{\mm_1}$ and $z \in R_{\mm_2}$ such that $y^p - y = x = z^p
    - z$.  Then $y-z$ is in the kernel of the Artin-Schreier map,
    which is $\Ff_p$, so $y \in z + \Ff_p \subseteq R_{\mm_2}$.  As $y
    \in R_{\mm_1}$, this implies $y \in R_{\mm_1} \cap R_{\mm_2} = R$.
    Thus, at least one Artin-Schreier root ($y$) is in $R$.  The other
    Artin-Schreier roots of $x$ are the elements of $y + \Ff_p$, which
    are all in $R$.
  \end{claimproof}
  Let $J = \mm_1 \cap \mm_2$.  This is the Jacobson radical of $R$.
  By Proposition~\ref{connecter}, $J = J^{00}$, and there are no definable
  subgroups of finite index.  Consider the sets
  \begin{gather*}
    \Delta = \{(x,i,j) \in R \times \Ff_p \times \Ff_p : x - i \in \mm_1, ~ x - j \in \mm_2\} \\
    \Gamma = \{(x^p - x, i - j) : (x,i,j) \in \Delta\}.
  \end{gather*}
  Then $(\Delta,+)$ and $(\Gamma,+)$ are definable groups.
  \begin{claim}
    $\Gamma$ is the graph of a group homomorphism $\psi$ from $(J,+)$
    onto $(\Ff_p,+)$.
  \end{claim}
  \begin{claimproof}
    First, we show that $\Gamma \subseteq J \times \Ff_p$.  Suppose
    $(x,i,j) \in \Delta$.  Then $x \equiv i \pmod{\mm_1}$, so $x^p - x
    \equiv i^p - i \equiv 0 \pmod{\mm_1}$, and $x^p - x \in \mm_1$.
    Similarly, $x^p - x \in \mm_2$, and therefore $x^p - x \in J$.
    Thus $(x^p - x, i - j) \in J \times \Ff_p$.

    Next we show that $\Gamma$ projects onto $J$.  Take $y \in J$.  By
    Claim~\ref{claim-pre} there is $x \in R$ with $x^p - x = y$.  Then
    $x^p - x \equiv y \equiv 0 \pmod{\mm_1}$, so $x^p - x \equiv i
    \pmod{\mm_1}$ for some $i \in \Ff_p$.  Similarly, $x^p - x \equiv
    j \pmod{\mm_2}$ for some $j \in \Ff_p$.  Then $(x,i,j) \in \Delta$
    and $(x^p - x, i - j) = (y, i - j) \in \Gamma$.

    Next we show that the projection $\Gamma \to J$ is one-to-one.
    Otherwise, $\Gamma \to J$ has non-trivial kernel, so there is
    $(x,i,j) \in \Delta$ with $x^p - x = 0$ but $i - j \ne 0$.  The
    fact that $x^p - x = 0$ implies $x \in \Ff_p$, and so $x \equiv i
    \pmod{\mm_1}$ implies $x = i$.  Similary, $x = j$.  But then $i -
    j = 0$, a contradiction.

    So now we see that $\Gamma \to J$ is one-to-one and onto, implying
    that $\Gamma$ is the graph of some group homomorphism $\psi$ from
    $J$ to $\Ff_p$.  It remains to show that $\psi$ is onto.
    Equivalently, we must show that $\Gamma$ projects onto $\Ff_p$.
    Let $i \in \Ff_p$ be given.  By the Chinese remainder theorem,
    there is $x \in R$ such that $x \equiv i \pmod{\mm_1}$ and $x
    \equiv 0 \pmod{\mm_2}$.  Then $(x,i,0) \in \Delta$, so $(x^p - x,
    i - 0) \in \Gamma$.  The element $(x^p - x,i)$ projects onto $i$.
    Equivalently, $\psi(x^p - x) = i$.
  \end{claimproof}
  Therefore there is a definable surjective group homomorphism $\psi :
  J \to \Ff_p$.  The kernel $\ker \psi$ is a definable subgroup of $J$
  of index $p$.  This contradicts Proposition~\ref{connecter}.
\end{proof}

\begin{lemma}\label{step-2}
  Let $R$ be a NIP integral $\Ff_p$-algebra.  Let $\pp_1, \pp_2$ be
  prime ideals such that $R/\pp_1$ and $R/\pp_2$ are infinite.  Then
  $\pp_1$ and $\pp_2$ are comparable.
\end{lemma}
\begin{proof}
  Suppose otherwise.  Let $S = R \setminus (\pp_1 \cup \pp_2)$.  Then
  $S$ is a multiplicative subset of $R$.  Let $R' = S^{-1} R$.  Then
  $R'$ is NIP by Corollary~\ref{localize-nip}.  The ring $R'$ has
  exactly two maximal ideals $\mm_1, \mm_2$, where $\mm_i = \pp_i R'$.
  The map $R/\pp_i \to R'/\mm_i$ is injective, so $R'/\mm_i$ is
  infinite, for $i = 1, 2$.  This contradicts Lemma~\ref{step-1}.
\end{proof}



\begin{lemma} \label{step-3}
  Let $R$ be an $\Ff_p$-algebra that is integral and has exactly two
  maximal ideals $\mm_1$ and $\mm_2$.  Then $R$ isn't NIP.
\end{lemma}
\begin{proof}
  Assume otherwise.  Going to an elementary extension, we may assume
  that $R$ is very saturated.  By the Chinese remainder theorem, there
  is some $a \in R$ such that $a \equiv 0 \pmod{\mm_1}$ but $a \equiv
  1 \pmod{\mm_2}$.

  Let $\Sigma(x)$ be the partial type saying that $x \in \mm_1$, $x
  \notin \mm_2$, and $x$ does not divide $a^n$ for any
  $n$.
  \begin{claim}
    $\Sigma(x)$ is finitely satisfiable.
  \end{claim}
  \begin{claimproof}
    Let $n$ be given.  We claim there is an $x$ such that $x \in
    \mm_1$, $x \notin \mm_2$, and $x$ does not divide $a^i$ for $i \le
    n$.  Take $x = a^{n+1}$.  Then $x \equiv 0^{n+1} \equiv 0
    \pmod{\mm_1}$, so $x \in \mm_1$.  But $x \equiv 1^{n+1} \equiv 1
    \pmod{\mm_2}$, so $x \notin \mm_2$.  Finally, suppose $x =
    a^{n+1}$ divides $a^i$ for some $i \le n$. Then there is $u \in R$
    with $ua^{n+1} = a^i$.  Since $R$ is a domain, we can cancel a
    factor of $a^i$ from both sides, and see $ua^{n+1-i} = 1$.  This
    implies that $a$ is a unit, contradicting the fact that $a \in
    \mm_1$.
  \end{claimproof}
  By saturation, there is $a' \in R$ satisfying $\Sigma(x)$.  The
  principal ideal $(a')$ does not intersect the multiplicative set $S
  := a^{\Nn}$, by definition of $\Sigma(x)$.  Let $\pp_1$ be maximal
  among ideals containing $(a')$ and avoiding $S$.  Then $\pp_1$ is a
  prime ideal.  (In general, any ideal that is maximal among ideals
  avoiding a multiplicative set is prime.)

  Now $\pp_1 \not \subseteq \mm_2$, because $a' \in \pp_1$ but $a'
  \notin \mm_2$.  But $\pp_1$ must be contained in \emph{some} maximal
  ideal, and so $\pp_1 \subseteq \mm_1$.  The inclusion is strict,
  because $a \in \mm_1$ but $a \notin \pp_1$.  Thus $\pp_1 \subsetneq
  \mm_1$ and $\pp_1 \not \subseteq \mm_2$.  In particular, $\pp_1$ is
  not a maximal ideal.

  Similarly, there is a non-maximal prime ideal $\pp_2$ with $\pp_2
  \subsetneq \mm_2$ and $\pp_2 \not \subseteq \mm_1$.  Then $\pp_1$
  and $\pp_2$ are incomparable.  Otherwise, say, $\pp_1 \subseteq
  \pp_2 \subsetneq \mm_2$, and so $\pp_1 \subseteq \mm_2$, a
  contradiction.  For $i = 1, 2$, the fact that $\pp_i$ is a
  non-maximal prime ideal implies that $R/\pp_i$ is a non-field
  integral domain, and therefore infinite.  This contradicts
  Lemma~\ref{step-2}.
\end{proof}

\begin{theorem} \label{linear}
  Let $R$ be a NIP integral $\Ff_p$-algebra.  Then the prime ideals of
  $R$ are linearly ordered by inclusion.
\end{theorem}
\begin{proof}
  The same proof as Lemma~\ref{step-2}, using Lemma~\ref{step-3}
  instead of Lemma~\ref{step-1}.
\end{proof}

\begin{corollary} \label{forest}
  Let $R$ be a NIP $\Ff_p$-algebra.  Let $\pp_1, \pp_2, \qq$ be prime
  ideals.  If $\pp_i \supseteq \qq$ for $i =1 , 2$, then $\pp_1$ is
  comparable to $\pp_2$.
\end{corollary}
\begin{proof}
  Otherwise, $\pp_1$ and $\pp_2$ induce incomparable primes in the NIP
  domain $R/\qq$.
\end{proof}

\subsection{Henselianity}
\begin{definition}
  A \emph{forest} is a poset $(P,\le)$ with the property that if $x
  \in P$, then the set $\{y \in P : y \ge x\}$ is linearly ordered.
\end{definition}
\begin{definition}
  A ring $R$ is \emph{good} if $\Spec R$ is a forest of finite width.
\end{definition}
\begin{lemma} \label{good-times}
  ~
  \begin{enumerate}
  \item If $R$ is a NIP $\Ff_p$-algebra, then $R$ is good.
  \item If $R$ is good, then any quotient $R/I$ is good.
  \item If $R$ is good, then $R$ is a finite product of local rings.
  \end{enumerate}
\end{lemma}
\begin{proof}
  ~
  \begin{enumerate}
  \item Fact~\ref{width} and Corollary~\ref{forest}.
  \item Clear, since $\Spec R/I$ is a subposet of $\Spec R$.
  \item We now break our usual convention, and regard $\Spec R$ as a
    scheme, or at least a topological space.  By scheme theory, it
    suffices to write $\Spec R$ as a finite disjoint union of clopen
    sets $U_i$, such that each $U_i$ contains a unique closed point.
    Let $\mm_1, \ldots, \mm_n$ be the maximal ideals of $R$.  There
    are finitely many because $\Spec R$ has finite width.  Note that
    every prime ideal $\pp \in R$ satisfies $\pp \subseteq \mm_i$ for
    a unique $i$.  (There is at least one $i$ by Zorn's lemma, and at
    most one $i$ because $\Spec R$ is a forest.)  Let $U_i$ be the set
    of primes below $\mm_i$.  Then $\Spec R$ is a disjoint union of
    the $U_i$.  It remains to show that each $U_i$ is clopen.  It
    suffices to show that each $U_i$ is closed.  Take $i = 1$.  Let
    $\pp_1, \ldots, \pp_m$ be the minimal primes contained in $\mm_1$.
    (There are finitely many, because of finite width.)  Let $V_j$ be
    the set of primes containing $\pp_j$.  Then $V_j$ is a closed
    subset of $\Spec R$---it is the closed subset cut out by the ideal
    $\pp_j$.  Moreover, $V_j \subseteq U_1$, because $\Spec R$ is a
    forest.  The sets $V_1, \ldots, V_m$ cover $U_1$, because every
    prime contains a minimal prime.  Then $U_1$ is a finite union of
    closed sets $\bigcup_{i = 1}^m V_i$, and so $U_1$ is
    closed. \qedhere
  \end{enumerate}
\end{proof}

\begin{proposition} \label{hensel-wow}
  Let $R$ be a NIP local $\Ff_p$-algebra.  Then $R$ is a Henselian
  local ring.
\end{proposition}
\begin{proof}
  By \cite[Lemma~04GG, condition (9)]{stacks-project}, it suffices to
  prove the following: any finite $R$-algebra is a product of local
  rings.  Let $S$ be a finite $R$-algebra.  Let $a_1, \ldots, a_n$ be
  elements of $S$ which generate $S$ as an $R$-module.  Each $a_i$ is
  integral over $R$, so there is a monic polynomial $P_i(x) \in R[x]$
  such that $P_i(a) = 0$ in $S$.  Then there is a surjective
  homomorphism
  \begin{equation*}
    R[x_1,\ldots,x_n]/(P_1(x_1),\ldots,P_i(x_i)) \to S.
  \end{equation*}
  The ring on the left is interpretable in $R$---it is a finite-rank
  free $R$-module with basis the monomials $\prod_{i = 1}^n x_i^{n_i}$
  for $\bar{n} \in \prod_{i = 1}^n \{0,1,\ldots,\deg P_i - 1\}$.
  Therefore, the left hand side is a NIP ring.  By
  Lemma~\ref{good-times}, it is good, $S$ is good, and $S$ is a finite
  product of local rings.
\end{proof}

\begin{theorem} \label{prod-of-hens}
  Let $R$ be a NIP $\Ff_p$-algebra.  Then $R$ is a finite product of
  Henselian local rings.
\end{theorem}
\begin{proof}
  By Lemma~\ref{good-times}, $R$ is good, and $R$ is a finite produt
  of local rings.  These local rings are easily seen to be
  interpretable in $R$, so they are also NIP.  By
  Proposition~\ref{hensel-wow}, they are Henselian local rings.
\end{proof}

\begin{theorem} \label{one-hens}
  Let $R$ be a NIP, integral $\Ff_p$-algebra.  Then $R$ is a Henselian
  local domain.
\end{theorem}
\begin{proof}
  $R$ is a local ring by Theorem~\ref{linear}.  Therefore it is
  Henselian by Proposition~\ref{hensel-wow}.
\end{proof}
Recall that a field $K$ is \emph{large} (also called \emph{ample}) if
every smooth irreducible $K$-curve with at least one $K$-point
contains infinitely many $K$-points \cite{Pop-little}.  By
\cite[Theorem 1.1]{Pop-henselian}, if $R$ is a henselian local domain
that is not a field, then $\Frac(R)$ is large.  Therefore we get the
following corollary:
\begin{corollary} \label{joke-conclusion}
  Let $R$ be a NIP integral domain, and $K = \Frac(R)$.  Suppose $R
  \ne K$ and $K$ has positive characteristic.  Then $K$ is large.
\end{corollary}
Large stable fields are classified \cite{firstpaper}.  If we could
extend this classification to large NIP fields, then
Corollary~\ref{joke-conclusion} would tell us something very strong
about NIP integral domains of positive characteristic.

\begin{acknowledgment}
  The author was supported by Fudan University, and by Grant
  No. 12101131 of the National Natural Science Fund of China.
\end{acknowledgment}

\bibliographystyle{plain} \bibliography{little-bib}{}

\end{document}